\documentclass{article}
\usepackage{amsthm}
\usepackage{amsmath}
\usepackage{amsfonts}
\usepackage{amssymb}
\usepackage{amscd}
\usepackage[all,cmtip]{xy}
\newtheorem{theorem}{Theorem}
\newtheorem{proposition}{Proposition}
\newtheorem{lemma}{Lemma}
\newtheorem*{corollary}{Corollary}
\DeclareFontEncoding{OT2}{}{} 
\newcommand{\textcyr}[1]{\fontencoding{OT2}\fontfamily{cmr}\fontseries{m}\fontshape{n}\selectfont #1}
\newcommand{\Sha}{\mbox{\textcyr{Sh}}}
\usepackage[T1]{fontenc}
\usepackage{ae,aecompl}

\begin{document}

\title{\huge Visibility of Ideal Classes}    
\author{Ren\'e Schoof \thanks{Dipartimento di Matematica, 2$^a$ Universit\`a di Roma ``Tor Vergata'',
I-00133 Roma, Italy, schoof@mat.uniroma2.it} \and
Lawrence C. Washington \thanks{Department of Mathematics, 
University of Maryland, 
College Park, MD 20742, lcw@math.umd.edu}}  
\maketitle
\begin{abstract} 
Cremona, Mazur, and others have studied what they call 
visibility of elements of Shafarevich-Tate groups of elliptic curves.
The analogue for an abelian number field $K$ is capitulation of
ideal classes of $K$ in the minimal cyclotomic field containing $K$.
We develop a new method to study capitulation and use it and 
classical methods to compute data
with the hope of gaining insight into the
elliptic curve case. For example, the numerical data for number fields
suggests that visibility of nontrivial
Shafarevich-Tate elements might be much more common for elliptic
curves of positive rank than for curves of rank 0. 
\end{abstract}
 
Let $E$ be an elliptic curve over $\mathbf  Q$ of conductor $N$. Then
there is a modular parametrization $X_0(N)\to E$ and a 
corresponding map $E\to J_0(N)$ of
$E$ into the jacobian of $X_0(N)$. This induces
a map of Shafarevich-Tate groups $\Sha  (E)\to \Sha  (J_0(N))$.
Cremona and Mazur \cite{cremona} 
study the question of when an element $s$ of $\Sha  (E)$ is in the kernel
of this map. When this happens, there is a curve defined over $\mathbf  Q$
and contained in $J_0(N)$ that represents $s$. In other words,
$s$ is ``visible'' in $J_0(N)$. Further work on this topic has been done
by Agashe, Stein, and others \cite{agashe}, \cite{agashe-stein}.

Let $K/\mathbf  Q$ be an abelian extension of conductor $n$, so 
$K\subseteq \mathbf  Q(\zeta_n)$, where $n$ is the conductor of $K$. 
This is the analogue
of the modular parametrization above. The ideal class
group is the analogue of the Shafarevich-Tate group (this will be made
more precise in Section 1), so the analogue of the above
question is to ask when ideal classes of $K$ become principal
in $\mathbf  Q(\zeta_n)$. 

Let $L/K$ be an extension of number fields. An ideal class of $K$ that
becomes principal in $L$ is said to {\it capitulate}. Many authors
have discussed capitulation in various contexts.
The hope of the present paper is to use results about capitulation
for $\mathbf  Q(\zeta_n)/K$ to gain some insight into the 
situation for Shafarevich-Tate groups.

For example, for imaginary quadratic fields $K$, the capitulation 
in $\mathbf  Q(\zeta_n)/K$
is mostly accounted for by the ambiguous classes, namely those
produced by genus theory, and such classes are easily seen to capitulate.
In examples of small discriminant, the class group consists
mostly of ambiguous classes, so capitulation is very common.
However, for large discriminants, most of the ideal class group
does not capitulate.

The situation for real quadratic fields is quite different.
The ambiguous classes capitulate, but
numerical data indicates that capitulation of additional
classes is very common. The situation for cyclic
cubic fields is similar.

Let's return to Shafarevich-Tate groups.
Cremona and Mazur looked at elliptic curves of 
conductor up to 5500. All of their examples of
non-trivial $\Sha  (E)$ had Mordell-Weil rank 0,
and capitulation (i.e., visibility)
was very common. The analogue of the Mordell-Weil
group is the unit group of the ring of integers of a number field.
An elliptic curve with Mordell-Weil rank 0 therefore corresponds
to an imaginary quadratic field (or $\mathbf  Q$). As pointed
out above, capitulation is very common for imaginary quadratic fields with small
discriminants. The numerical results
of Cremona and Mazur for elliptic curves match this situation. By analogy,
it is natural to ask the following questions.
Should we expect non-visibility to become the 
rule rather than the exception
as the conductor increases? Is there an analogue of genus theory
that accounts for most of the Shafarevich-Tate group
for small conductors? We do not know the answer to either
question. Both 
seem worth investigating.

Real quadratic fields (along with non-real cubics 
and totally complex quartics)
correspond to elliptic curves with Mordell-Weil rank 1.
Cyclic cubic fields correspond to elliptic curves of rank 2.
The results for real quadratic fields and cyclic cubic fields
suggest that, for elliptic curves
of positive rank, visibility should be very common. 
There is very little data available; again, this seems to be worth
investigating.

The Cohen-Lenstra heuristics \cite{cohen} predict that, for totally real
fields, the existence of units makes class numbers tend to be small.
Therefore, it is perhaps common that the class number of a totally
real abelian number field 
has approximately 
the same class number as the minimal real cyclotomic field containing it.
This tends to force capitulation of ideal classes (see part (i) of
Lemma 4 in Section 4).
Is there is an analogous situation for elliptic curves of positive rank?
Of course, there is a difference in this case between the number field
case and the elliptic curve case. When a number field is properly 
contained in a cyclotomic field, the rank
of the group of units of the number field is less than
that of the cyclotomic field (except when the number field is the real subfield of the cyclotomic field). There does not seem to be a reason
for an analogous situation for the Mordell-Weil
ranks of an elliptic curve and the corresponding $J_0(N)$.

\section{The analogy between ideal class groups and Shafarevich-Tate groups}

In this section we review and make explicit some analogies between
ideal class groups and Shafarevich-Tate groups.

Let $E$ be an elliptic curve defined over $\mathbf  Q$. The Shafarevich-Tate
group is defined to be the group of everywhere locally
trivial elements of $H^1(G_{\mathbf  Q}, E(\overline{\mathbf  Q}))$, namely
$$
\Sha  (E)=\mbox{\it Ker}\bigg(H^1(G_{\mathbf  Q},E(\overline{\mathbf  Q}))\to \prod_{p\le \infty}
H^1(G_{\mathbf  Q_p}, E(\overline{\mathbf  Q}_p))\bigg).
$$
Here $G_L=\mbox{Gal}(\overline{L}/L)$ for any field $L$, 
and embeddings $\overline{\mathbf  Q}
\hookrightarrow \overline{\mathbf  Q}_p$ are implicitly fixed.

Now let's consider the number field case. Let $K$ be a number field. 
Since units are the
analogue of points on an elliptic curve, we let $E$ denote
the group of all units in the ring of integers of $\overline{\mathbf  Q}$.
For a place $\mathfrak p$ of $K$, let $K_{\mathfrak p}$ be the completion 
at $\mathfrak p$. When $\mathfrak p$ is finite, let $p$ denote the rational
prime below $\mathfrak p$
and let $U_p$ be the group of units in the integral closure
of $\mathbf  Z_p$ in $\overline{\mathbf  Q}_p$.
When $\mathfrak p$ is archimedean, let $U_p=\mathbf  C^{\times}$.  Let $C_K$ denote the ideal
class group of $K$. 
The following result has been implicit in the literature
(for example, apply the techniques of
\cite[pp. 98-99]{milne} to the inclusion
$j$: Spec $K$ $\to$ Spec $O_K$ for the \'etale sheaf $G_m$)
and also appears in \cite{flach} and \cite[Lemma 6.1]{gonzalez}.
Since we need it and its proof later, 
we include it.
It shows the strong analogy between ideal class groups
and Shafarevich-Tate groups. The cohomology groups are the standard
profinite cohomology groups defined using continuous cocycles.

\begin{proposition} There is an isomorphism
$$ C_K\simeq \mbox{Ker}\bigg(H^1(G_K,E)\to \prod_{\mathfrak p}
H^1(G_{K_{\mathfrak p}}, U_p)\bigg).
$$
\end{proposition}
\begin{proof} The product is over the places of $K$.
Since $H^1(G_{\mathbf  R}, \mathbf  C^{\times})=0$ and 
$H^1(G_{\mathbf  C}, \mathbf  C^{\times})=0$,
the archimedean primes can be ignored in the statement of the theorem and
in the following.

 Let $I$ be an ideal of $K$. Then $I$ becomes principal
in some extension of $K$, so $I=(\beta)$ for some $\beta\in \overline{\mathbf  Q}$
(we use $I$ to denote the lift of $I$ to extension fields).
Define $c_I:\mbox{Gal}(\overline{\mathbf Q}/K)\to E$ by $\sigma\mapsto \beta^{\sigma-1}$. 
Since  $\sigma(I)=I$, we have $\beta^{\sigma-1}\in E$. It follows easily that
$c_I$ is a continuous cocycle. If also $I=(\beta_1)$, 
then $\epsilon=\beta/\beta_1\in E$, so the cocycle defined using $\beta_1$
differs from $c_I$ by the coboundary $\sigma\mapsto \epsilon^{\sigma-1}$.
Therefore the cohomology class of $c_I$ depends only on $I$.
In fact, $c_I$ depends only on the ideal class of $I$: if $a\in K^{\times}$
then 
$$
c_{aI}(\sigma)=(a\beta)^{\sigma-1}=\beta^{\sigma-1}=c_I(\sigma).
$$
Therefore we have a homomorphism $\phi:C_K\to H^1(G_K, E)$.

Suppose $c_I$ is a coboundary, which means there is 
some $\epsilon\in E$ such that
$\beta^{\sigma-1}=\epsilon^{\sigma-1}$ for all $\sigma$. This means that
$\beta/\epsilon\in K$, so
$I=(\beta/\epsilon)$ is principal in $K$. Therefore $\phi$ is injective.

We now show that $c_I$ is locally trivial. Fix a prime ideal $\mathfrak p$
of $K$. Choose an ideal $J$ in the ideal class of $I$ with $J$ prime to
$\mathfrak p$. Then $J=(\gamma)$ for some $\gamma\in \overline{\mathbf  Q}$,
and $\gamma\in U_p$ via the embedding of $\overline{\mathbf  Q}\hookrightarrow 
\overline{\mathbf  Q}_p$ induced by $\mathfrak p$. Therefore $c_J$ restricted to
$G_{K_{\mathfrak p}}$ is given by the coboundary $\sigma\mapsto \gamma^{\sigma-1}$,
so the cohomology class of $c_I=c_J$ is locally trivial.

Finally, we show $\phi$ is surjective. Let $c$ be a cocycle in $H^1(G_K,E)$ 
that is locally a coboundary.
Hilbert's Theorem 90 says that $H^1(G_K,\overline{\mathbf  Q}^{\times})=0$. The map
$E\to \overline{\mathbf  Q}^{\times}$ therefore sends $c$ to a coboundary, so 
$c(\sigma)=y^{\sigma-1}$ for some $y\in \overline{\mathbf  Q}^{\times}$.
Since $c$ is continuous, $c$ has finite order, hence $c^m$ 
is a coboundary for some $m$.
Therefore, $c^m(\sigma)=\epsilon^{\sigma-1}$ for some $\epsilon\in E$.
This implies that $y^m/\epsilon$ is fixed by all $\sigma$, hence is in $K$.
Let $\alpha=y^m/\epsilon$.

Let $\mathfrak p$ be a prime ideal of $K$. Since $c$ is a coboundary at $\mathfrak p$,
we have $c(\sigma)=u_{\mathfrak p}^{\sigma-1}$ for all $\sigma\in G_{K_{\mathfrak p}}$
for some $u_{\mathfrak p}\in U_p$. Therefore,
$y^{\sigma-1}=u_{\mathfrak p}^{\sigma-1}$ for all $\sigma\in G_{K_{\mathfrak p}}$,
hence $y/u_{\mathfrak p}\in K_{\mathfrak p}$. Therefore
$$
v_{\mathfrak p}(\alpha)=v_{\mathfrak p}(y^m)=v_{\mathfrak p}((y/u_{\mathfrak p})^m)\equiv 0\pmod m.
$$
Since this happens for all $\mathfrak p$, we must have $(\alpha)=I^m$ for some ideal $I$
of $K$. 

Let $\epsilon_1=\epsilon^{1/m}\in E$. Then
$I=(y/\epsilon_1)$ in some extension of $K$. It follows easily that the
cohomology class of $c_I$ equals the cohomology class of $c$. Therefore
$\phi$ is surjective. 
\end{proof}

Let $L/K$ be a finite Galois extension of number fields and let $E_L$
be the units of the ring of integers of $L$. Define the locally trivial cohomology
group to be
$$
H_{lt}^1(L/K, E_L)=\mbox{\it Ker}\bigg(H^1(\mbox{Gal}(L/K),E_L)\to \prod_{\mathfrak p}
H^1(\mbox{Gal}(L_{\mathfrak p}/K_{\mathfrak p}), U_{L_{\mathfrak p}})\bigg),
$$
where $L_{\mathfrak p}$ denotes the completion of $L$ at one of the primes of $L$ above
$\mathfrak p$ and $U_{L_{\mathfrak p}}$ is the group of local units in $L_{\mathfrak p}$.

The inclusion map $K\hookrightarrow L$ induces a map $C_K\to C_L$. The following 
result appears in \cite{schmithals}.

\begin{corollary} There is an isomorphism  
$$
\mbox{Ker}(C_K\to C_L)\simeq H_{lt}^1(L/K, E_L).
$$
\end{corollary}

\begin{proof} The beginning of the inflation-restriction exact sequence is
$$
0\to H^1(L/K, E_L)\to H^1(G_K, E)\to H^1(G_L, E).
$$
An element $x\in H_{lt}^1(L/K, E_L)$ clearly yields an element of $H^1(G_K, E)$ that is
locally trivial, hence corresponds to an element $y\in C_K$. The map from
$H^1(G_K, E)$ to $H^1(G_L, E)$, when restricted to locally trivial elements, 
is easily seen to correspond to the map on class groups.
Since $x$ is 0 in $H^1(G_L, E)$, it follows that $y$ is 0 in $C_L$.
Therefore we have a map $\psi: H_{lt}^1(L/K, E_L)\to \mbox{\it Ker}(C_K\to C_L)$.
The injectivity of $\psi$ is immediate from the injectivity on the left
in the inflation-restriction sequence. 

It remains to show that $\psi$ is surjective. An element of $\mbox{\it Ker}(C_K\to C_L)$
corresponds to a cocycle $c$ whose class in $H^1(L/K,E_L)$  is locally trivial when 
regarded as an element $\tilde{c}\in H^1(G_K,E)$. We must show that $c$ is locally trivial
in $H^1(L/K,E_L)$. Since $\tilde{c}$ is the inflation of $c$, we have $\tilde{c}(\sigma)=1$
for all $\sigma\in G_L$.
Let $\mathfrak p$ be a prime ideal
of $K$. The local triviality in $H^1(G_K,E)$ implies that there 
exists $u_{\mathfrak p}\in U_{\mathfrak p}$ such that
$\tilde{c}(\sigma)=u_{\mathfrak p}^{\sigma-1}$ for all $\sigma\in G_{K_{\mathfrak p}}$. 
Since $u_{\mathfrak p}^{\sigma-1} =\tilde{c}(\sigma)=1$ for all $\sigma\in G_L
\cap G_{K_{\mathfrak p}}=G_{L_{\mathfrak p}}$,
we have $u_{\mathfrak p}\in U_{L_{\mathfrak p}}$. This means that $c\in H_{lt}^1(L/K, E_L)$.
Therefore $\psi$ is surjective. 
\end{proof}

In the case of elliptic curves, the fundamental descent sequence
is
$$
0\to E(\mathbf  Q)/nE(\mathbf  Q) \to S_n \to \Sha  [n]\to 0, 
$$
where $S_n\subseteq H^1(G_{\mathbf  Q}, E[n])$ is the $n$-Selmer group.
There is an analogue for number fields. Recall that, for a number field $K$,
$$
H^1(G_K,\mu_n)\simeq K^{\times}/(K^{\times})^n.
$$
This follows easily from Hilbert's Theorem 90. 

\begin{proposition} Let $K$ be a number field and let $n\ge 1$. Let
$$
S_n=\{x\in K^{\times}|\, (x)=I^n \mbox{ for some ideal } I \subset K\}\big/(K^{\times})^n.
$$
Then there is an exact sequence
$$
1\to E_K/(E_K)^n\to S_n\to C_K[n]\to 1,
$$
where $C_K[n]$ denotes the $n$-torsion in $C_K$ and where the map from
$S_n$ to $C_K[n]$ sends $x$ to the ideal class of $I$.
Also, $S_n$ is the inverse image of $H^1_{lt}(G_K,E_K)[n]$ under
the maps
$$
K^{\times}/(K^{\times})^n \simeq H^1(G_K,\mu_n)\to H^1(G_K,E_K).
$$
\end{proposition}
\begin{proof} The exactness of the sequence is straightforward.
To verify the last claim, use the fact that 
$g\mapsto g(x^{1/n})/x^{1/n}$ gives the cocycle in $H^1(G_K,\mu_n)$ corresponding to $x$.
Moreover, $I=(x^{1/n})$ in $K(x^{1/n})$, so this is also the cocycle in  
$H_{lt}^1(G_K,E)$ 
corresponding to the ideal class of $I$ under the isomorphism 
of Proposition 1. 
\end{proof}

{\bf Remark.} In the elliptic curve situation, we are interested in 
$$
\mbox{\it Ker}\bigg(\Sha  (E)\to \Sha  (J_0(N))\bigg)\subseteq \mbox{\it Ker}\bigg(H^1(G_{\mathbf  Q}, 
E(\overline{\mathbf  Q}))\to H^1(G_{\mathbf  Q}, J_0(N)(\overline{\mathbf  Q}))\bigg).
$$
In the number field case, we have
$$
\mbox{\it Ker}(C_K\to C_L)\hookrightarrow  \mbox{\it Ker}\bigg(H^1(G_K, E)\to H^1(G_L, E)\bigg).
$$
This map is obtained from the map on Galois groups rather than a map on $E$,
which would be closer to the geometric situation. However, this can easily be remedied.
Let 
$$
J_{L}=\mbox{Hom}_{G_L}(\mathbf  Z[G_{\mathbf  Q}], E)= \mbox{ Maps}(G_{\mathbf  Q}/G_L, E).
$$ 
Shapiro's Lemma says that
$H^1(G_{\mathbf  Q},J_{L})\simeq H^1(G_L, E)$. Therefore,
we have
$$
\mbox{ \it Ker}(C_K\to C_L)\hookrightarrow  \mbox{\it Ker}\left(H^1(G_{\mathbf  Q}, J_K)\to H^1(G_{\mathbf  Q}, J_{L})\right),
$$
and the map is obtained from the natural map $J_K\hookrightarrow J_{L}$.

{\bf Remark.} It is known (see \cite{cremona}) that $\Sha(E)$ becomes trivial in some
abelian variety containing $E$. The question has been raised whether there are classes
of $\Sha(E)$ that are not visible in $J_0(N)$ but which become
trivial in $J_0(M)$ for some $M$ that is a multiple of $N$.
The analogous question can be asked for number fields. The following example shows that
this situation can arise.
Consider the cubic subfield $K$ of $\mathbf Q(\zeta_{163})$.  
It was shown by Kummer that the class group of $K$ is the product of two
groups of order 2. Since $[\mathbf Q(\zeta_{163})^+:K]=27$, and since
the map on ideal classes from the real subfield to the full cyclotomic 
field is injective, these classes
do not capitulate in $\mathbf Q(\zeta_{163})$. In \cite{gras}, G. Gras
points out that the ideal class group of $K$ capitulates in $K(\sqrt{13})$.
Since $\sqrt{13}\in \mathbf Q(\zeta_{13})$, the ideal class group of $K$
capitulates in $\mathbf Q(\zeta_{13\cdot 163})$.

\section{Creating visible elements}

One of the methods Cremona and Mazur use for identifying
visible elements of Shafarevich-Tate groups is the following
(\cite[p. 19]{cremona}). Let $E$ and $E'$ be  elliptic curves contained
in $J_0(N)$ and assume $E[n]=E'[n]$ as subgroups of $J_0(N)$.
Suppose $c\in \Sha(E)\subseteq H^1(G_{\mathbf  Q}, E(\overline{\mathbf  Q}))$
is the image of $\xi\in H^1(G_{\mathbf  Q}, E[n])$. If $\xi$ maps to
$0\in H^1(G_{\mathbf  Q}, E'(\overline{\mathbf  Q}))$, then $c$ maps to 
$0\in \Sha(J_0(N))$, hence is visible. 
$$
\xymatrix{
&c\in H^1(G_{\mathbf Q}, E)\ar[rrdd] && \\
\xi \in H^1(G_{\mathbf Q}, E[n]) \ar[ru]\ar[rd] &  && \\
&  H^1(G_{\mathbf Q}, J_0(N)[n]) \ar[rr] && H^1(G_{\mathbf Q}, J_0(N)) \\
\xi \in H^1(G_{\mathbf Q}, E'[n]) \ar[ru]\ar[rd] &  && \\
& 0\in H^1(G_{\mathbf Q}, E') \ar[rruu] &&
}
$$
In this situation, there exists $R\in E'(\overline{\mathbf Q})$ such that
$\sigma(R)-R=\xi(\sigma)$ for all $\sigma\in \text{Gal}(\overline{\mathbf Q}/\mathbf Q)$.
Since $\xi(\sigma)\in E'[n]$, we have $\sigma(nR)=nR$ for all $\sigma$, so
$nR\in E'(\mathbf Q)$.

When $n=2$, this situation is very much analogous to
the fact that ideal classes of order 2 in quadratic
fields capitulate in suitably chosen
biquadratic fields. Let $K=\mathbf  Q(\sqrt{d})$ be a quadratic
field of discriminant $d$ and let $1<d_1<|d|$ be a fundamental discriminant
dividing $d$ (if such a $d_1$ exists).  Let $F=\mathbf  Q(\sqrt{d_1})$ and $L=K(\sqrt{d_1})$. Then 
$L/K$ is an unramified quadratic extension.
Moreover, $L$ is contained in the smallest cyclotomic field containing
$K$. 

The number $d_1\in K^{\times}/(K^{\times})^2\simeq H^1(K,\mu_2)$ yields an ideal
$J$ of $K$ with $J^2=(d_1)$. Since $d_1$ represents the trivial class in
$F^{\times}/(F^{\times})^2$, it also represents the trivial class in
$L^{\times}/(L^{\times})^2$, so we recover the obvious fact that $J$ becomes principal
in $L$.

A more interesting example is obtained as follows.
Let $\epsilon$ be the fundamental unit of $F$. Then, in 
the notation of Proposition 2,
$\epsilon$ represents a nontrivial element of $S_2(F)$ that has trivial image
in $C_F$.  Assume that the norm
of $\epsilon$ is $+1$. Let $\alpha=1+\epsilon^{-1}$. The 
norm of $\alpha$ is $a=(1+\epsilon^{-1})(1+\epsilon)
\in \mathbf  Z$. Let $\sigma$ be the nontrivial element of
$\mbox{Gal}(F/\mathbf  Q)$. Then $\alpha^{\sigma}/\alpha=\epsilon$.
Therefore the ideal $(\alpha)$ of $F$ is fixed by $\mbox{Gal}(F/\mathbf  Q)$,
and therefore also by $\mbox{Gal}(L/K)$. Since $L/K$ is unramified,
there is an ideal $I$ of $K$ such that $I=(\alpha)$ in $L$.
Moreover, $\alpha^2\epsilon=\alpha\alpha^{\sigma}=a\in \mathbf  Z$,
which implies that $I^2=(a)$ in $K$. The coset of 
$$
a=\alpha^2\epsilon 
\in K^{\times}/(K^{\times})^2\simeq H^1(G_K, \mu_2)
$$ 
maps to the coset of 
$$
\epsilon\in L^{\times}/(L^{\times})^2\simeq H^1(G_L,\mu_2).
$$
This coset maps to the trivial
ideal class in $C_L$, corresponding to the fact that the ideal class of $I$ 
capitulates in $L$. This is clearly an analogue of the 
elliptic curve situation described above, where $\sqrt{\epsilon}$ corresponds 
to the point $R$.

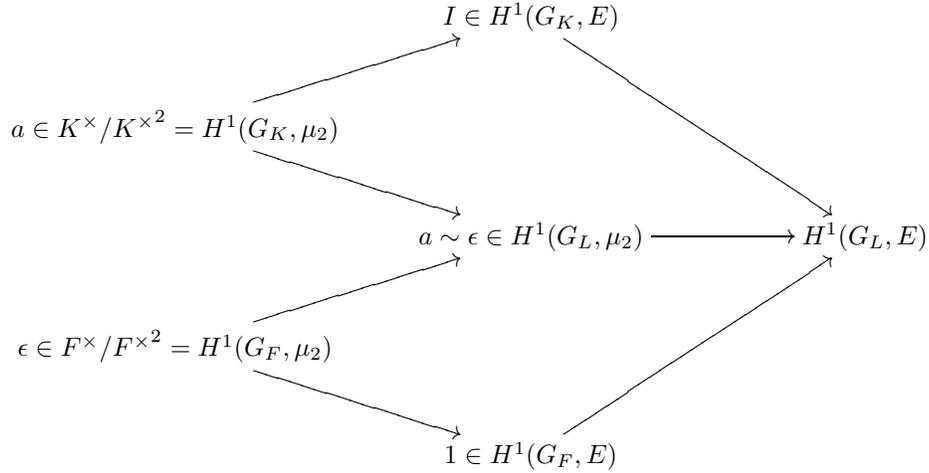
$$
\xymatrix{
&I\in H^1(G_K, E)\ar[rrdd] && \\
a\in K^{\times}/{K^{\times}}^2 = H^1(G_K, \mu_2) \ar[ru]\ar[rd] &  && \\
& a\sim \epsilon\in H^1(G_L, \mu_2) \ar[rr] && H^1(G_L,E) \\
\epsilon \in F^{\times}/{F^{\times}}^2 = H^1(G_F, \mu_2) \ar[ru]\ar[rd] &  && \\
& 1\in H^1(G_F, E) \ar[rruu] &&
}
$$

It is natural to ask about the class of the ideal $I$
just constructed. It can be identified via the following result.

\begin{lemma} Let $\epsilon=(x+y\sqrt{d_1})/2$ be the fundamental unit 
of $\mathbf  Q(\sqrt{d_1})$. Then
$x+2= rw^2$ for some positive integers $r, w$ such that $r|2d_1$ and 
such that $r$ and $4d_1/r$ are not squares.
\end{lemma}
\begin{proof} We have $x^2-d_1y^2=4$, so $(x+2)(x-2)=d_1y^2$. If $x+2$
or $x-2$
is a square, then $\sqrt{\epsilon}=\frac12(\sqrt{x+2}+\sqrt{x-2})\in
\mathbf  Q(\sqrt{d_1})$, which is a contradiction.
Since $\gcd(x+2, x-2)$ divides 4, the result follows easily.
\end{proof}

Since $I^2=(a)$ with $a=(1+\epsilon^{-1})(1+\epsilon)=2+x=rw^2$
in the notation of the lemma, we find that $(I/w)^2=(r)$. Therefore
the ideal class of $I$ comes from the class of $r$ in
$K^{\times}/(K^{\times})^2$. The ideal class of $I$ is nontrivial in
$K$ but capitulates in $L$. This 
capitulation represents ``non-obvious'' capitulation in $L$ (where
the ``obvious'' capitulation is for the ideal whose square is
$(d_1)$). Moreover, the factorization of $(1+\epsilon)$, which equals $I/w$ in $L$,
into a product
of primes gives the ``non-obvious'' relation in the class group
of $F$ (where the ``obvious'' relation is that the product of all
ramified primes, with a possible omission of the prime above 2, is principal).

\section{Capitulation of ideal classes: General results}

\begin{lemma} Let $K\subseteq L$ be number fields with $[L:K]=n$. Let $I$ be an ideal
of $K$. If $I$ becomes principal in $L$ then $I^n$ is principal in $K$. \end{lemma}
\begin{proof}
If $I$ is principal in $L$, its norm is principal in $K$. But the norm is $I^n$. 
\end{proof}

{\bf Remark.} The analogue of this result is true for elliptic curves:
If an element of $\Sha(E)$ is visible, then its order divides
the degree of modular parametrization of $E$ (see \cite{cremona}).

\medskip

Let $K/\mathbf Q$ be an abelian extension of degree $d$
and of conductor $n$, so $K\subseteq \mathbf Q(\zeta_n)$. 
We say that an ideal class has {\it potential capitulation} if
its order divides $\phi(n)/d$. We say that $K$ has {\it maximal
capitulation} if all ideal classes with potential capitulation
actually capitulate in $\mathbf  Q(\zeta_n)$,
and {\it maximal $p$-capitulation} if all classes of $p$-power order with
potential capitulation actually capitulate in $\mathbf Q(\zeta_n)$.

The relevant question to consider is whether classes with potential
capitulation actually capitulate. There is a marked difference in the behaviors
of the real and imaginary ideal classes. The real case is related to the result
of Kurihara \cite{kurihara} (see also \cite{gras}) that says that if $K$ is totally real,
then all ideal classes of $K$ capitulate in the field $K(\zeta_{\infty})$ obtained 
by adjoining all roots of unity to $K$. On the other hand, a result of
Brumer \cite{brumer} says that the class group (defined as a direct limit) 
of the extension of $\mathbf Q$ generated by 
all roots of unity is isomorphic to a countable direct sum of factors $\mathbf Q/\mathbf Z$.
By Kurihara's result, these classes cannot arise from class groups of real fields. 
Of course, both of these results relate to capitulation in fields much
larger than the minimal cyclotomic field containing $K$. But they give
an indication of the difference between the two cases.

There is an explanation of why there should be a lot of capitulation in the
real case. The Cohen-Lenstra heuristics \cite{cohen} predict that class numbers of
real fields tend to be small. Suppose a prime $p$ divides the class number of
$K\subseteq \mathbf Q(\zeta_n)^+$. Since $h_K$ tends to be small, it is likely that
$p^2\nmid h_K$. Now suppose that $p$ divides the class number $h_n^+$ of the
real cyclotomic field. By the same reasoning, it is
likely that $p^2\nmid h_n^+$. By Lemma 4 (i) below, the classes of order $p$
in $K$ capitulate in this case. We shall see examples of this phenomenon in Section 6.

The following result is useful when working with totally real fields $K$ of prime
conductor $\ell$. It shows that we need to consider capitulation only from $K$ to the real
subfield $\mathbf Q(\zeta_{\ell})^+$ of the cyclotomic field.

\begin{proposition} Let $\ell$ be prime. Then the map from the class group
of $\mathbf Q(\zeta_{\ell})^+$ to the class group
of $\mathbf Q(\zeta_{\ell})$ is injective.
\end{proposition}
\begin{proof} See \cite[Thm. 4.14]{washington}).
\end{proof}

\section{Classical methods}

To treat the imaginary classes, we need the following. 

\begin{lemma} Let $K$ be a number field contained in the $n$th cyclotomic field
$\mathbf  Q(\zeta_n)$, and let $d$ be the number of roots of unity in $K$. Let
$\mu_m$ be the group of roots of unity in $\mathbf Q(\zeta_n)$ (so $m=n$ or $2n$).
Then $H^1(G_{\mathbf  Q(\zeta_n)/K},\, \mu_m)$ is annihilated by $d$.
\end{lemma}
\begin{proof} Let $G$ be a group, let $t\in G$, and let $A$ be a $G$-module.
The automorphism $g\mapsto tgt^{-1}$ gives $A$ a new module structure;
call it $A^t$. The map $\psi: a\mapsto t^{-1}a$ is a $G$-homomorphism
from $A^t$ to $A$. Proposition 3
of \cite{atiyah} says that the composite map
$$
H^1(G,A)\to H^1(G,A^t)\stackrel{\psi_*}{\to} H^1(G,A)
$$
is the identity map.

In our case, identify $\mbox{Gal}(\mathbf  Q(\zeta_n)/K)$ with a subgroup $G$ 
of $(\mathbf  Z/m\mathbf  Z)^{\times}$.
The module $A=\mu_m$ becomes $\mathbf  Z/m\mathbf  Z$ with $G$ acting by multiplication.
Since conjugation by $t\in G$ is trivial, $A^t=A$. The map $\psi$
is multiplication by an integer $t'\equiv t^{-1}\pmod m$, so 
the map on cohomology is also multiplication by $t'$.
Therefore 
$$
t': H^1(G_{\mathbf  Q(\zeta_n)/K},\, \mu_m)\to H^1(G_{\mathbf  Q(\zeta_n)/K},\, \mu_m)
$$
is the identity for all integers $t'\in G$. 
Let $d=\gcd(\{t'-1\},m)$, where $t'$ runs through all such integers.
Then $d$ annihilates $H^1(G_{\mathbf  Q(\zeta_n)/K},\, \mu_m)$.

If $\zeta\in \mu_m$, then $\zeta\in K$ if and only if 
$\zeta^{t'}=\zeta$ for all $t'\in G$. Therefore $\zeta\in K$ 
if and only if
$\zeta^d=1$. This means that there are exactly  $d$ roots
of unity in $K$. 
\end{proof}

\begin{proposition} Let $K$ be a subfield of $\mathbf  Q(\zeta_n)$ and let
$d$ be the number of roots of unity in $K$. Let $I$ be an ideal
of $K$ that becomes principal in $\mathbf  Q(\zeta_n)$.
Then $(I/\overline{I})^d$ is principal in $K$ (where $\overline{I}$ denotes the 
complex conjugate).
\end{proposition}
\begin{proof} Suppose $I=(\alpha)$ in $\mathbf  Q(\zeta_n)$. Let $\sigma\in
\mbox{Gal}(\mathbf  Q(\zeta_n)/K)$. Then $\alpha^{\sigma-1}\in E_{\mathbf  Q(\zeta_n)}$. Therefore,
$(\alpha/\overline{\alpha})^{\sigma-1}$ is a unit of absolute value 1, hence
a root of unity. It follows that the map
$\sigma\mapsto (\alpha/\overline{\alpha})^{\sigma-1}$ is a cocycle
for $H^1(G_{\mathbf  Q(\zeta_n)/K},\, \mu_m)$, where $\mu_m$ is the group of all roots of unity
in $\mathbf Q(\zeta_n)$. By Lemma 3, there exists $\zeta\in \mu_m$ such that 
$(\alpha/\overline{\alpha})^{d(\sigma-1)}=\zeta^{\sigma-1}$ for all $\sigma$.
This implies that $(\alpha/\overline{\alpha})^d/\zeta\in K$, hence that
$(I/\overline{I})^d$ is principal in $K$. 
\end{proof}

\begin{corollary} Let $K$ be a subfield of $\mathbf Q(\zeta_n)$ 
and suppose that
the only roots of 
unity in $K$ are $\pm1$. If $I$ is an ideal of $K$ such that $I\overline{I}$
is principal in $K$ and $I$ is principal in $\mathbf  Q(\zeta_n)$, then
$I^4$ is principal in $K$. 
\end{corollary}

The corollary applies in particular to imaginary quadratic fields.
In this case, the 4-torsion of the class group is of order at most
$4^{s-1}$, where $s$ is the number of prime factors of the discriminant.
It is therefore an easy consequence of Siegel's theorem (that $\log(h)\sim
\frac12\log(|d|)$) that the 4-torsion is only a small part
of the class group when the discriminant is large, and therefore
most of the class group does not capitulate in $\mathbf  Q(\zeta_d)$.

\smallskip
{\bf Example.} Here is an example where a class of order 4 capitulates.
Let $K=\mathbf  Q(\sqrt{-39})$, whose class group is cyclic of
order 4, generated by the ideal $I=(2,\frac{1+\sqrt{-39}}{2})$.
The ideal $J=(3,\frac{-3+\sqrt{-39}}{2})$ is not principal
but satisfies $J^2=(3)$, hence has order 2 in the class
group of $K$. Therefore $I^2$ is in the same class
as $J$. Since $J$ becomes principal in $L=\mathbf  Q(\sqrt{-39}, \sqrt{-3})$,
namely $J=(\sqrt{-3})$, it follows that $I^2$ is principal in 
$L$. However, $I$ cannot be principal in $L$ since then
the norm from $L$ to $K$ of $I$, namely $I^2$, would
be principal in $K$, which is not the case.
The class number of $\mathbf  Q(\zeta_{39})$ is 2.
Since $\mathbf  Q(\zeta_{39})/L$ is totally ramified, the norm $N: C_{\mathbf  Q(\zeta_{39})}\to C_L$
is surjective. 
If $I$ is not principal in $\mathbf  Q(\zeta_{39})$, then it generates the class group,
hence $N(I)=I^6$ generates the class group of $L$; contradiction.
Therefore $I$ is principal in $\mathbf  Q(\zeta_{39})$.
\smallskip

For a quadratic field, the classes of order 2 always capitulate
in the cyclotomic field.
This is easily seen as follows: Let $K=\mathbf  Q(\sqrt{d})$, where 
$d$ is the discriminant of $K$.
The ideal classes of order 2
are generated by ideals  $I$ with $I^2=(r)$ and $r$ 
dividing $d$. However, $\mathbf  Q(\zeta_{|d|})$ is the smallest cyclotomic field
containing $K$, and $\mathbf  Q(\zeta_{|d|})$ contains $\sqrt{\pm r}$ for each
such $r$ and an appropriate choice of sign. Therefore each $I$ becomes
principal in $\mathbf  Q(\zeta_{|d|})$.

More generally, Furuya \cite{furuya} has shown that when $K/\mathbf  Q$ is an abelian
extension,  every ideal fixed by $\mbox{Gal}(K/\mathbf  Q)$ becomes
principal in the genus field of $K$ (i.e., the maximal abelian
extension of $\mathbf  Q$ that is unramified over $K$). Since
the genus field is contained in the smallest cyclotomic field
containing $K$, all such ideals capitulate in the cyclotomic field.

The following result is useful in many cases.
\begin{lemma}  Suppose $\ell$ is prime and $L/F$ is an 
extension of number fields. Also, suppose $L/F$ has no nontrivial
unramified subextensions $M/F$. Let $h_F$ and $h_L$ be the class numbers of $F$ and $L$. \newline
(i)  Assume $[L:F]=\ell^a$. If $\ell\nmid h_L/h_F$, then 
the kernel of the map
$C_F\to C_L$ is exactly the classes of order dividing $\ell^a$.\newline
(ii) Assume $\ell$ is odd and $L/F$ is Galois of degree $\ell$.
If the $\ell$-power part of the class
group of $L$ is cyclic of order $\ell^k$ and the $\ell$-power part of
the class group of $F$ has order $\ell^f$ with $f<k$, then
$f=k-1$ and the map $C_F\to C_L$ is injective.\newline
(iii) Assume $\ell$ is odd and $L/F$ is Galois of degree $\ell$.
If the $\ell$-power part of the class group of $L$ is isomorphic to
$\mathbf Z/\ell\mathbf Z\times \mathbf Z/\ell\mathbf Z$, then all classes
of order $\ell$ in $F$ become principal in $L$.

\end{lemma}
\begin{proof} (i)  The only possible capitulation occurs
in the $\ell^a$-torsion. Let $A_F$ and $A_L$ be the $\ell$-power parts
of the class groups. Since the norm map from $A_L$ to $A_F$ is surjective,
it must be an isomorphism since the groups have the same order.
Let $I$ represent a class in $A_F$ of order dividing $\ell^a$.
Lifting $I$ to $L$ then taking the norm back to $F$ yields $I^
{\ell^a}$, which is principal. Since the norm is injective,
the image of $I$ in $C_L$ must have been trivial.

(ii) The case $k=1$ is trivial, so assume $k>1$.
The map $C_F\to C_L$ is injective on the non-$\ell$ parts,
so we restrict our attention to $A_F$ and $A_L$. By assumption,
$A_L\simeq \mathbf  Z/\ell^k\mathbf  Z$. A generator $\sigma$ of $\mbox{Gal}(L/F)$
acts on $\mathbf  Z/\ell^k\mathbf  Z$ as an automorphism of order 1 or $\ell$,
hence by multiplication by $1+a\ell^{k-1}$ for some $a$. It follows
easily that $N=1+\sigma+\sigma^2+\cdots +\sigma^{\ell-1}$ acts
as multiplication by $\ell$. But $N$ is the composition of the two
maps
$$
A_L\to A_F \to A_L,
$$
where the first map is the norm and the second is the natural map
on class groups.
Since the image of $N$ has order $\ell^{k-1}$ and $A_F$ is
assumed to have order
at most this large, $A_F$ has order exactly $\ell^{k-1}$. 
Also, $A_F\to A_L$ must be an injection.

(iii) Let $\sigma$ generate $\text{Gal}(L/F)$. Then $\sigma$ is a linear transformation of 
the $\ell$-part of the class group of $L$, which is 
a $\mathbf Z/\ell \mathbf Z$ vector space. The Jordan canonical 
form of $\sigma$ is $M=\begin{pmatrix} 1 & b \\ 0 & 1\end{pmatrix}$ for some $b$.
Therefore, 
$$
1+\sigma+\cdots +\sigma^{\ell-1} = I + M + \cdots +M^{\ell-1}\equiv 0\pmod {\ell},
$$
so $1+\sigma+\cdots +\sigma^{\ell-1}$ annihilates the classes of $L$ of order $\ell$. 
But $1+\sigma+\cdots +\sigma^{\ell-1}$ is the norm map followed by the natural map from the class group
of $F$ to the class group of $L$. Since $L/F$ is totally ramified, the norm map is surjective.
Therefore, all classes of order $\ell$ in $F$ must become principal in $L$. 
\end{proof}

\section{Galois module methods}

In this section, we introduce a method based on \cite{schoofplus} that is much more
powerful than those of the previous section when working with totally real abelian fields of prime conductor. 

Let $p$ be prime, let $\ell\equiv 1\pmod p$ also be prime, and let 
$G=\text{Gal}(\mathbf Q(\zeta_{\ell})^+/\mathbf Q)$. Let $\pi$ be
the maximal subgroup of $G$ of $p$-power order and $\Delta$ the maximal subgroup
of $G$ of order prime to $p$. Then $G=\pi\times\Delta$. Let $p^n$ be the order
of $\pi$.

Let
$\chi: \Delta\to \overline{\mathbf Q}_p^{\times}
$ be a $p$-adic valued Dirichlet character and let ${\cal O}_{\chi}$ be the extension
of $\mathbf Z_p$ generated by the values of $\chi$. It is a $\mathbf Z[\Delta]$-algebra
via $\delta\cdot x=\chi(\delta)x$ for $\delta\in\Delta$ and $x\in {\cal O}_{\chi}$. For
any $\mathbf Z[G]$-module $M$, define its $\chi$-eigenspace to be
$$
M(\chi)=M\otimes_{\mathbf Z[\Delta]}\, {\cal O}_{\chi}.
$$
The functor $M\mapsto M(\chi)$ is exact. We have that
$$
M\simeq \prod_{\chi}M(\chi),
$$
where $\chi$ runs through representatives for the 
$\text{Gal}(\overline{\mathbf Q}_p/\mathbf Q_p)$-conjugacy classes of characters
of $\Delta$. The eigenspace $M(\chi)$ has the natural structure of an
${\cal O}_{\chi}[\pi]$-module. Moreover,
$$
{\cal O}_{\chi}[\pi]\simeq {\cal O}_{\chi}[[T]]/(\omega_n(T)),
$$
where $1+T$ corresponds to the choice of a generator of $\pi$
and $\omega_n(T)=(1+T)^{p^n}-1$.

Let $F$ be the fixed field of $\pi$. So $\Delta$ is identified 
with $\text{Gal}(F/\mathbf Q)$. Let $H\subseteq \Delta$ be the kernel 
of $\chi$ and let $K\subseteq F$ be the fixed
field of $H$. 

Let $A_K$ denote the $p$-Sylow subgroup of the class group of $K$, and similarly for
other fields.

\begin{lemma} The natural map $A_K\to A_F$ yields an isomorphism
$A_K(\chi)\simeq A_F(\chi)$.\end{lemma}

\begin{proof}  Let $N: A_F(\chi)\to A_K(\chi)$ be the norm map. The natural map
$A_K(\chi)\to A_F(\chi)$ followed by $N$ is the $|H|$-th power map. Since $p\nmid |H|$,
this is an injection, so $A_K(\chi)\to A_F(\chi)$ is injective. Since $\chi(H)=1$,
the map $A_F(\chi)\to A_F(\chi)$ given by $N$ followed by the natural map
from $A_K(\chi)$ to $A_F(\chi)$ is the $|H|$-th power map, so
$A_K(\chi)\to A_F(\chi)$ is surjective. 
\end{proof}

{\bf Remark.} The lemma holds more generally if $K$ is replaced by any field
between $K$ and $F$.

In the lemma, we can, for example, take $\chi$ to be trivial.
We find that $A_F(1)$ is trivial. Henceforth, we assume that $\chi\ne 1$. 
If we take $\chi$ to be quadratic 
of conductor $\ell\equiv
1\pmod 4$ and let $p$ be odd, we find
that $A_F(\chi)$ is isomorphic to the $p$-Sylow subgroup of the class group
of $K=\mathbf Q(\sqrt{\ell})$. 

\begin{proposition} Let 
$$
V=\text{Ker}\left(A_F\to A_{\mathbf Q(\zeta_{\ell})^+}\right).
$$
If $\chi\ne 1$, then
$$
V(\chi)\simeq H^1(\text{Gal}(\mathbf Q(\zeta_{\ell})^+/F),\, 
E_{\mathbf Q(\zeta_{\ell})^+})(\chi).
$$
\end{proposition}
\begin{proof} Let $\mathfrak p$ be a prime of $F$ and let $\mathfrak q$
be a prime of $\mathbf Q(\zeta_{\ell})^+$ above $\mathfrak p$.
The exact sequence $0\to U_{\mathbf Q(\zeta_{\ell})_{\mathfrak q}^+} 
\to \mathbf (Q(\zeta_{\ell})^+_{\mathfrak q})^{\times}
\to \mathbf Z\to 0$ yields the exact sequence
$$
K_{\mathfrak p}^{\times} \to \mathbf Z \to H^1(\pi, U_{\mathbf
Q(\zeta_{\ell})^+_{\mathfrak q}})\to 0.
$$
Since the image of the valuation map on $K_{\mathfrak p}^{\times}$ is $e\mathbf Z$,
we see that 
$$H^1(\pi, U_{\mathbf
Q(\zeta_{\ell})^+_{\mathfrak q}})\simeq \mathbf Z/e\mathbf Z,$$
where $e$ is the ramification index of $\mathfrak q$ over $\mathfrak p$, and the 
action of $\Delta$
on this cohomology group is trivial.

We know from the Corollary to Proposition 1 that $V$ is given by 
locally trivial cohomology
classes. 
Since $\mathbf Q(\zeta_{\ell})^+/F$ is ramified only at $\ell$,
and the ramification index there is $p^n$,
we have  
$$
V\simeq \text{\it Ker}\left(H^1(\pi, E_{\mathbf Q(\zeta_{\ell})^+})\to \mathbf Z/p^n\mathbf Z)\right),
$$
The result follows. 
\end{proof}

Let $Cycl$ denote the cyclotomic units of $\mathbf Q(\zeta_{\ell})^+$, namely,
the group generated by elements of the form $(\zeta_{\ell}^a-\zeta_{\ell}^{-a})
/(\zeta_{\ell}^b-\zeta_{\ell}^{-b})$. Let 
$$
B=E_{\mathbf Q(\zeta_{\ell})^+}/Cycl.
$$
Let $I$ be the augmentation ideal of  $\mathbf Z[G]$. 
The exact sequence 
$$
0\to I\to \mathbf Z[G]\to \mathbf Z\to 0
$$
implies that there is an isomorphism of $\Delta$-modules
$$
\widehat{H}^q(\pi, I)\simeq \widehat{H}^{q-1}(\pi, \mathbf Z)
$$
for all $q$, where $\widehat{H}$ denotes a Tate cohomology group.
Since $G$ is commutative and $\pi$ and $\Delta$ have coprime orders,
the group $\Delta$ acts trivially on $\widehat{H}^{q-1}(\pi, \mathbf Z)$
and on $\widehat{H}^q(\pi, \{\pm 1\})$ (\cite[Lemma 1.1]{schoofminus}).
Therefore, if $\chi\ne 1$, 
$$
\widehat{H}^q(\pi, I)(\chi)\simeq \widehat{H}^{q-1}(\pi, \mathbf Z)(\chi)=0$$
and $\widehat{H}^{q}(\pi,\{\pm 1\})(\chi)=0$ for all $q$.
The inverse of the map $I\to Cycl/\{\pm 1\}$ given by
$$
\sigma -1\mapsto \frac{\sigma(\zeta-\zeta^{-1})}{\zeta-\zeta^{-1}}
$$
yields an exact sequence (see \cite[Proposition 8.11]{washington})
$$
0\to \{\pm 1\}\to Cycl \to I \to 0.      
$$
This implies that $\widehat{H}^1(\pi, Cycl)(\chi)=0$ for all $q$.
It follows, when $\chi\ne 1$, that
$$
V(\chi)\simeq H^1(\pi, E_{\mathbf Q(\zeta_{\ell})^+})(\chi)\simeq H^1(\pi, B)(\chi).
$$

For a finite Galois module $M$, let $M^d=\text{Hom}_{\mathbf Z}(M,\mathbf Q/\mathbf Z)$ be 
the dual of $M$. The Galois action is given by $(\sigma f)(m)=f(\sigma^{-1}m)$,
so the pairing between $M$ and $M^d$ induces a nondegenerate pairing
between $M(\chi)$ and $M^d(\chi^{-1})$.
Duality theory (\cite{iyanaga}) tells us that 
that $H^1(\pi, B)(\chi)=H^1(\pi, B(\chi))$ is dual 
to $H_1(\pi, B^d(\chi^{-1}))$. This latter group
equals $\widehat{H}^{-2}(\pi, B^d(\chi^{-1}))$, which is isomorphic to 
$\widehat{H}^0(\pi, B^d(\chi^{-1}))$. Therefore, 
$$
|V(\chi)| = |\widehat{H}^0(\pi, B^d(\chi^{-1}))|.
$$

Note that $B^d(\chi^{-1})$ is a module over 
${\cal O}_{\chi}[\pi]\simeq {\cal O}_{\chi}[[T]]/
(\omega_n(T))$. In \cite{schoofplus}, it is shown how to use cyclotomic units to
compute an ideal 
$I_{\chi^{-1}}\subseteq {\cal O}_{\chi} [[T]]$ such that 
$$
B^d(\chi^{-1})\simeq {\cal O}_{\chi}[[T]]/I_{\chi^{-1}}.
$$
\begin{theorem}
Let $\chi\ne 1$. Then
$V(\chi)$ is dual to
$$
\{ f\in {\cal O}_{\chi}[[T]] \; \big| \; Tf\in I_{\chi^{-1}}\}/\left(I_{\chi^{-1}}
+(\omega_n(T)/T)\right).
$$
\end{theorem}
\begin{proof}
Since $1+T$ is a generator of $\pi$, an element $f\in B^d(\chi^{-1})$ is 
fixed by $\pi$ if and only if $Tf=0$.  The norm for $\pi$ is given by
$\omega_n(T)/T$. Since $\widehat{H}^0$ is given by fixed elements modulo
norms, the result follows. 
\end{proof}

\begin{corollary} All of $A_K\simeq A_F(\chi)$ capitulates in $\mathbf Q(\zeta_{\ell})^+$
if and only if $\omega_n(T)/T\in I_{\chi^{-1}}$.
\end{corollary}
\begin{proof} Let $B_F$ be the units of $F$ modulo the cyclotomic units
of $F$. It is known (\cite{mw}) that $|A_F(\chi)|=|B_F(\chi)|$.
But $B_F(\chi)$ is isomorphic to the $\pi$-invariant subgroup of $B(\chi)$ (\cite[Prop. 5.1 (i)]{schoofplus}).
Let $\sigma$ generate $\pi$.
Under the pairing between $B(\chi)$ and $B^d(\chi^{-1})$, the annihilator
of $(1-\sigma)B^d(\chi^{-1})$ is the $\pi$-invariant subgroup of $B(\chi)$.
Therefore, $B_F(\chi)$ is dual to, and therefore has the same order as, 
$B^d(\chi^{-1})/(1-\sigma)B^d(\chi^{-1})$, 
which is the maximal 
quotient of $B^d(\chi^{-1})$ on which $\pi$
acts trivially. This is isomorphic to ${\cal O}_{\chi}[[T]]/(I_{\chi^{-1}}+(T))$,
which has the same order as $\{f\in{\cal O}_{\chi}[[T]] \, | \, Tf\in 
I_{\chi^{-1}}\}/I_{\chi^{-1}}$
(proof: since ${\cal O}_{\chi}[[T]]/I_{\chi^{-1}}$ is finite, the kernel and cokernel of multiplication
by $T$ have the same order). 
Therefore, the order of $A_F(\chi)$ equals the order of $V(\chi)$ 
if and only if $\omega_n(T)/T\in I_{\chi^{-1}}$.
\end{proof}

The ideals $I_{\chi^{-1}}$ for small $p$ and for $\ell < 10000$ are listed in 
\cite[Tables 4.3 and 4.4]{schoofplus}.
We give three examples of how to apply Theorem 1.

\medskip
\noindent{\bf Example 1.} $l=2089$, $p=3$, and $\chi$ quadratic. In this case
$\pi$ has order~9 and the ring $O_{\chi}$ is equal to $\mathbf Z_3$. 
In \cite[Table 4.4]{schoofplus},
we find that $I_{\chi^{-1}}=(T-3,27)$. This implies that the 3-part of the class group
of the quadratic field has order~3, namely, the order of 
${\cal O}_{\chi}[[T]]/(I_{\chi^{-1}}+(T))$. 
We have that
$T\cdot f\in I_{\chi^{-1}}$ if and only if $9$ divides $f(3)$. These power series form
an ideal of index 9 in $\mathbf Z_3[[T]]$. On the other hand, ${{\omega_2(T)}/T}$
is congruent to $(4^9-1)/3=9\cdot 3059$ modulo~$I_{\chi^{-1}}$, so that the ideal
$I_{\chi^{-1}}+({{\omega_2(T)}/ T})$ has index 9 as well. It follows that the module
in Theorem~1 is {\it trivial}. Therefore there is no capitulation.

\medskip
\noindent{\bf Example 2.} $l=7489$, $p=2$ and $\chi$ is cubic. In this case $\pi$
has order 32 and the ring $O_{\chi}$ is the ring~$\mathbf Z_2[\zeta]$ where $\zeta$
denotes a cube root of unity. By \cite[Table 4.4]{schoofplus}, 
the ideal $I_{\chi^{-1}}$ is equal to
$(T+2+4\zeta,\, 8)$. This implies that the 2-part of the class group of the cubic
field has order~4. Since 
$$
{\omega_5(T)}/T\equiv ((-1-4\zeta)^{32}-1)/(-2-4\zeta)\equiv 0\pmod{I_{\chi^{-1}}},
$$ 
all classes capitulate.

\medskip
\noindent{\bf Example 3.} $l=9337$, $p=2$ and $\chi$ is cubic. In this case $\pi$
has order 4 and the ring $O_{\chi}$ is as in the previous example. 
By \cite[Table 4.4]{schoofplus}, the
ideal $I_{\chi^{-1}}$ is equal to $(T+4-2\zeta,8)$. Once again the 2-part of the class
group of the cubic field has order~4. We have that
$T\cdot f\in I_{\chi^{-1}}$ if and only if $4$ divides $f(2\zeta)$. These power series
form an ideal of index 16 in~$O_{\chi}[[T]]$. On the other hand,
${{\omega_2(T)}/ T}$ is congruent to
$((-3+2\zeta)^4-1)/(-4+2\zeta)\equiv 4\zeta\pmod{I_{\chi^{-1}}}$.
Therefore the ideal $I_{\chi^{-1}}+({{\omega_2(T)}/ T})$ also has index~16, the
module of Theorem~1 is trivial, and there is no capitulation.

\section{Numerical data}

\subsection{Imaginary quadratic fields}

Consider imaginary quadratic fields 
$\mathbf  Q(\sqrt{-d})$, where $-d$ is the discriminant.
There
are 31 examples with $d<100$. Of these, 8 have trivial class groups,
14 have nontrivial class groups that completely capitulate,
and the remaining 9 fields have elements in their class groups that do not
capitulate. As pointed out above, for sufficiently large $d$
there will always be some elements of the class group 
that do not capitulate. In fact, most elements will not capitulate.

\subsection{Real quadratic fields; $\ell=3$}

In Section 4, we saw that capitulation was fairly predictable
for imaginary quadratic fields. It was restricted to the 4-torsion
of the class group, and a lot of it could be explained by genus theory.
The situation for real quadratic fields seems to be entirely different.

The ideal classes of order 2 capitulate in the cyclotomic field,
as we showed in Section 4. However, Proposition 4
gives us no additional information. In fact, it seems that
there is no easy way to predict when capitulation occurs.

Consider the 3-parts of the class groups for
fields $K=\mathbf  Q(\sqrt{\ell})$, where $\ell\equiv 1\pmod {4}$ is prime.

There are 52 primes $\ell<10000$ with $\ell\equiv 5\pmod{12}$ and such that
the 3-class group of $\mathbf  Q(\sqrt{\ell})$ is nontrivial. Since $3\nmid \phi(\ell)/2$,
the ideal classes of order 3 do not capitulate.

On the other hand, when $\ell\equiv 1\pmod{12}$, the ideal classes of order 3 (and 
sometimes those of orders 9, 27, ...) have potential
capitulation. 

Here are the results of some computations.

\begin{center}
\begin{tabular}{rl}
\multicolumn{2}{l}{\rule[-2mm]{0mm}{2mm}\!\!$\mathbf  Q(\sqrt{\ell}), \quad \ell\equiv 1\pmod{12}, \quad \ell<10000$}\\
\hline\rule[5mm]{0mm}{0mm}
32:& \quad  \mbox{ non-trivial 3-class group}\\
26:& \quad \mbox{ maximal 3-capitulation}\\
6:&\quad \mbox{ no 3-capitulation}
\end{tabular}
\end{center}

There seems to be no reason to expect that these are
small exceptional cases. In fact, a suitable modification of the
Cohen-Lenstra heuristics presumably would predict that
certain cases of capitulation (covered by part (i) of Lemma 4) 
and of non-capitulation (covered 
by part (ii) of Lemma 4)
occur with positive density. 

A reasonable prediction from the data is that capitulation
is fairly common, and probably fairly random, for
cases of potential capitulation. 

Here are the details of the computations. They were carried out in PARI.

Part (i) of Lemma 4 shows that there is maximal 3-power capitulation
for the following primes: 
$$\begin{array}{cccccccccc}
\phantom{2}229,& \phantom{2}733,& 1129,& 1489,& 2557,& 2677,& 2713,& 2857,& 2917,& 3877,\\
3889,&  4597,& 4729,& 5521,& 5821,& 
6133,& 6997,& 7057,& 7537,& 7573,\\
7753,& 8713,& 9133.\end{array}
$$
All of the capitulation in these examples occurs in the sextic
subfield of $\mathbf  Q(\zeta_\ell)$.

Part (ii) of Lemma 4 shows that there is no 3-power capitulation
for the following primes:
$$ 3229,\quad 5281,\quad 6637,\quad 8017,\quad 8581.
$$
In each of these examples, the 3-part of the 
class group of the sextic field
$L$ is cyclic and $3\nmid [\mathbf  Q(\zeta_\ell):L]$,
so no capitulation can occur from $L$ to $\mathbf Q(\zeta_{\ell})$.

There are four primes remaining: 2089, 4933, 7873, 8761.
They can be treated by the methods of Section 5. We find the
following:

{\bf 2089}: The class group of $K$ is $\mathbf  Z/3\mathbf  Z$. There is no capitulation. 

{\bf 4933}: The class group of $K$ is $\mathbf  Z/3\mathbf  Z$. It capitulates
in the cyclotomic field.

{\bf 7873}: The class group of $K$ is $\mathbf  Z/9\mathbf  Z$. The capitulation kernel
has order 3, which is maximal capitulation.

{\bf 8761}: The class group of $K$ is $\mathbf  Z/27\mathbf  Z$. The capitulation kernel
has order 3, which is maximal capitulation.

In all of the above examples, the class group of $K$ is cyclic,
and the capitulation, when it occurs, is maximal.
For a non-cyclic class group, consider $\ell=114889$.
The class group is $\mathbf  Z/3\mathbf  Z\times \mathbf  Z/3\mathbf  Z$. 
A calculation with PARI shows that the capitulation in the sextic field,
hence in $\mathbf  Q(\zeta_\ell)$, has order 3. Therefore only part of
the potential capitulation is actual capitulation in this case.

The majority of the above cases have $\ell\equiv 1\pmod{12}$ and $3\| h_K$.
When 3 exactly divides the class number $h_6$ of the sextic subfield of
$\mathbf Q(\zeta_{\ell})$, we are guaranteed
to have capitulation. Computations for $\ell<500000$ yield the following:

\begin{center}
\begin{tabular}{cccc}
\multicolumn{4}{c}{
\rule[-2mm]{0mm}{2mm}\!\!\! $K=\mathbf Q(\sqrt{\ell})$, $\ell\equiv 1\pmod {12}$, and $\ell<500000$}\\
\hline
\rule[5mm]{0mm}{0mm}
Condition on class number:  & $3|h_K$ & $3\| h_K$ & $3\| h_6$\\
Number of fields: & 1343 & 1181 & 787\\
\end{tabular}
\end{center}
This agrees with the philosophy stated earlier as to why capitulation is common
for totally real fields.

\subsection{Cyclic cubic fields; $\ell=2$}

The situation for cyclic cubic fields is similar to that for real quadratic
fields. Potential capitulation is very often actual
capitulation, except for an obstruction that we describe below,
We examined the 611 cyclic cubic fields $K$ of prime conductor $\ell<10000$.
Of these cubic fields, 505 have trivial class groups. Of 
the remaining 106, the
most common (61 occurrences) 
class group is $\mathbf  Z_2\times \mathbf  Z_2$. Since the 2-rank
of the class group must be even (see \cite[Thm. 10.8]{washington}) and the class
number is prime to 3 (see \cite[Thm. 10.4]{washington}), this is the smallest possible
non-trivial class group. We therefore
start by looking at the 2-primary part of the class group.

There are 69 fields with nontrivial 2-primary part of the class group.
All classes of order 2 have potential capitulation since
$\ell-1$ is even. Only three cases considered have classes of
order 4. One of these has $\ell\equiv 3\pmod 4$ and therefore
there is no capitulation (see below). The remaining two primes
(1777 and 4297) have class groups $\mathbf  Z_4\times \mathbf  Z_4$
and these classes have potential capitulation. All 
classes capitulate for 1777, and the capitulation
is $\mathbf  Z_2\times \mathbf  Z_2$ for 4297.

An interesting phenomenon arises. If $\ell\equiv 3\pmod 4$,
then $\mathbf  Q(\zeta_\ell)^+/K$ has odd degree, so no class
of even order capitulates in this subextension. Moreover, the map from the class group
of $\mathbf  Q(\zeta_\ell)^+$ to that of $\mathbf  Q(\zeta_\ell)$ is injective (see
\cite[Thm. 4.14]{washington}). Therefore, the 2-part of the
class group of $K$ does not capitulate in $\mathbf  Q(\zeta_\ell)$.
Of the 69 fields, 34 have $\ell\equiv 3\pmod 4$, hence there is
no capitulation in the 2-part. An interesting question is whether there
is an elliptic curve analogue of this phenomenon.

Of the remaining 35 fields (those with $\ell\equiv 1\pmod 4$), 
six have none of the 
2-part of the class group capitulate, 28 have the entire
2-part capitulate, and one ($\ell=4297$) has partial capitulation. 
The following summarizes the computations:

\begin{center}
\begin{tabular}{rl}
\multicolumn{2}{l}{
\rule[-2mm]{0mm}{2mm}\!\!\! $[K:\mathbf  Q]=3, \quad K\subset \mathbf  Q(\zeta_\ell), \quad \ell<10000$}\\
\hline
\rule[5mm]{0mm}{0mm}
69: &\quad non-trivial 2-class group\\
34: &\quad $\ell\equiv 3\pmod 4$, therefore no 2-capitulation\\
28: &\quad $\ell\equiv 1\pmod 4$, maximal 2-power capitulation\\
1: &\quad $\ell\equiv 1\pmod 4$, partial 2-capitulation\\
6: &\quad $\ell\equiv 1\pmod 4$, no 2-capitulation\\
\end{tabular}
\end{center}

For 23 of the fields
with $\ell\equiv 1\pmod 4$, the 2-parts of the class groups
of both the cubic field $K$ and the sextic subfield of $\mathbf  Q(\zeta_\ell)$ are 
$\mathbf  Z_2\times \mathbf  Z_2$. Part (i) of 
Lemma 4 implies that the 2-part of the class group of
$K$ capitulates. 

The remaining cases are treated by the methods of Section 5.

As in the quadratic case, we extended these calculations to count, for $\ell < 500000$ and $\ell\equiv 1\pmod{12}$,
how often $4\| h_K$ (so the 2-part of the class group is $\mathbf Z/2\mathbf Z\times \mathbf Z/2\mathbf Z$)
and $4\| h_L$. In this case, capitulation is guaranteed by Lemma 4 (i). 
This situation accounts for a majority of
the cases where capitulation occurs.

\begin{center}
\begin{tabular}{cccc}
\multicolumn{4}{c}{
\rule[-2mm]{0mm}{2mm}\!\!\! $[K\, : \, \mathbf Q]=3$, $\ell\equiv 1\pmod{12}$, 
$\ell<500000$}\\
\hline
\rule[5mm]{0mm}{0mm}
Condition on class number: & $2|h_K$ & $4\| h_K$ & $4\| h_L$ \\
Number of fields: & 1447 & 1328 & 933 \\
\end{tabular}
\end{center}
Again, this agrees with the philosophy stated earlier as to why capitulation is common
for totally real fields. 

There are 24 primes less than 10000 where the 7-part of the class group
is non-trivial. Of these, 21 have $\ell\not\equiv 1\pmod 7$,
therefore do not have potential capitulation. The remaining
3 cases have maximal capitulation. The most interesting case is 7351, which
has class group $\mathbf  Z_{49}$ and can be treated by the method of Section 5.

For the remaining cubic fields $K$ with nontrivial class group,
there is no potential capitulation for any ideal classes. This is not surprising
since, for example, we expect only 1/18 of the primes with
19 in the class number of $K$ to have $\ell\equiv 1\pmod {19}$,
and there are not enough examples (7 in the case of 19)
for this to be very likely.

\subsection{Real quadratic fields; $\ell=5$}

There are 259 primes $\ell<500000$ such that $\ell\equiv 1\pmod {20}$ and the class number
of $\mathbf Q(\sqrt{\ell})$ is a divisible by 5. We have the following data:

\begin{center}
\begin{tabular}{rccc}
\multicolumn{4}{c}{
\rule[-2mm]{0mm}{2mm}\!\!\! $K=\mathbf Q(\sqrt{\ell})$, $\ell\equiv 1\pmod{20}$, 
$\ell<500000$}\\
\hline
\rule[5mm]{0mm}{0mm}
259:& \text{non-trivial 5-class group} \\
227:& \text{maximal 5-capitulation} \\
32:& \text{no capitulation}
\end{tabular}
\end{center}

For $\ell<400000$, the quadratic fields with non-trivial 5-class group were found, and then the class group
of the degree 10 subfield of $\mathbf Q(\zeta_{\ell})$ was computed using PARI. In these 203 cases, there
are 168 where the class group of both the quadratic field and the tenth degree field have 5-class group
cyclic of order 5. Therefore, Lemma 4 implies that all classes of order 5 capitulate. Moreover, of the
203 fields with $\ell<400000$, there are 147 such that the entire class group (not just the 5-part)
of the quadratic field is isomorphic to the class group of the tenth degree field.
This agrees with the philosophy that high degree totally real fields tend to have small
class numbers.

For $400000<\ell<500000$, the class number calculations started taking a long time, so the methods
of Section 5 were used to determine the capitulation behavior. They were also used for
the smaller prime $\ell=154501$ to determine that its 5-class group capitulates in $\mathbf Q(\zeta_{154501})$.
These methods are much faster than computing the full class number. This is to be expected, since
we need only the 5-part of the class number (however, calculating the 5-part of the units
modulo cyclotomic units could be used to compute the 5-part of the class number quickly, thus yielding an alternative approach
for the present purposes).

It seems worth mentioning a few details of computing the tenth degree field. For a prime $\ell\equiv 1\pmod {10}$,
the polynomial of the fifth degree subfield of $\mathbf Q(\zeta_{\ell})$ can be computed using formulas
of Tanner and Lehmer \cite{lehmer}. This yields a degree 5 polynomial $P(X)$. The following lemma shows that
$P(X+\sqrt{\ell})P(X-\sqrt{\ell})$, which can be computed numerically to sufficient accuracy and then rounded to
a polynomial with integral coefficients, is then the irreducible tenth degree polynomial that gives the desired field.
This method works well for small primes, but is slow for large primes $\ell$, say $\ell>400000$.

\begin{lemma} Let $L/K$ be a Galois extension of fields such that Gal$(L/K)$
is cyclic of order $mn$ with $\gcd(m,n)=1$.
Let $F_1$ be the subfield of degree $m$ over $K$ and let $F_2$ be the subfield of degree $n$ over $K$.
Write $F_1=K(\alpha)$ and $F_2=K(\beta)$. 
Then $L=K(\alpha+\beta)$. \end{lemma}
\begin{proof} Since $n=[K(\alpha)(\alpha+\beta):K(\alpha)]$ divides  $[K(\alpha+\beta):K]$, and similarly $m$ divides this
degree, the degree is at least $mn=[L:K]$.\end{proof}

\section*{Acknowledgments} 
The authors would like to thank Niranjan Ramachandran and Jonathan Rosenberg 
for helpful comments.


\begin{thebibliography}{10}

\bibitem{agashe} A. Agashe, ``On invisible elements of the Tate-Shafarevich group,''  
{\it C. R. Acad. Sci. Paris S\'er. I Math.} 328 (1999), no. 5, 369-374.
\bibitem{agashe-stein} A. Agashe and W. Stein, ``Visibility of Shafarevich-Tate groups of abelian varieties,'' {\it J. Number Theory} 97 (2002), 171-185.
\bibitem{atiyah} M. Atiyah and C. Wall, ``Cohomology of groups,'' {\it Algebraic number theory
(ed. by Cassels and Fr\"ohlich)}, Academic Press, 1967, pp. 94-115.
\bibitem{brumer} A. Brumer, ``The class group of all cyclotomic integers,'' {\it J. Pure Appl. Algebra}
20 (1981), 107--111.
\bibitem{cohen} H. Cohen and H. W. Lenstra,  ``Heuristics on class groups of number fields,''
{\it Number Theory, Noordwijkerhout 1983}, Springer Lecture Notes in Math. 1068 (1984), 33-62.
\bibitem{cremona} J. Cremona and B. Mazur, ``Visualizing elements in the 
Shafarevich-Tate group,'' {\it Experimental Math.} 9 (2000), 13-28.
\bibitem{flach} M. Flach, ``A generalisation of the {C}assels-{T}ate pairing,''
J. reine angew. Math. 412 (1990), 113-127.
\bibitem{furuya} H. Furuya, ``Principal ideal theorems in the genus field for
absolutely abelian extensions,'' {\it J. Number Theory} 9 (1977), 4-15.
\bibitem{gonzalez} C. Gonzalez-Aviles, ``Finite modules over non-semisimple group rings,'' {\it Israel J. Math.},
144 (2004), 61--92.
\bibitem{gras} G. Gras, ``Principalisation d'id\'eaux par extensions 
absolument ab\'eliennes,''
{\it J. Number Theory} 62 (1997), 403-421.
\bibitem{iyanaga} S. Iyanaga, {\it The theory of numbers}, North-Holland, 1975.
\bibitem{kurihara} M. Kurihara, ``On the ideal class groups of the maximal real subfields of
number fields with all roots of unity,'' {\it J. Eur. Math. Soc.} 1 (1999), 35-49.
\bibitem{lehmer} E. Lehmer, ``The quintic character of {$2$} and {$3$},''
{\it Duke Math. J.}, 18 (1951), 11-18.
\bibitem{mw} B. Mazur and A. Wiles, "Class fields of abelian extensions of $\mathbf Q$,''
{\it Invent. Math.}, 76 (1984), 179-330.
\bibitem{milne} J. Milne, ``The Tate-\v Safarevi\v c group of 
a constant abelian variety,'' {\it Invent. Math.} 6 (1968), 91--105. 
\bibitem{schmithals} B. Schmithals, ``Kapitulation der Idealklassen und Einheitenstruktur in
Zahlk\"orpern,'' {\it J. reine angew. Math.} 358 (1985), 43-60.
\bibitem{schoofminus} R. Schoof, Minus class groups of the fields of the $l$-th roots of
unity, {\it Math. Comp.} 67, no. 223 (1998), 1225--1245.
\bibitem{schoofplus} R. Schoof, ``Class numbers of real cyclotomic fields
of prime conductor,'' {\it Math. Comp.} 72, no. 242 (2003), 913-937.
\bibitem{serre} J-P. Serre, ``Local class field theory,''
{\it Algebraic number theory (ed. by Cassels and Fr\"ohlich)}, Academic Press, 1967, pp. 128-161.
\bibitem{washington} L. Washington, {\it Introduction to cyclotomic fields}, 
Springer-Verlag, 1987.
\end{thebibliography}
\end{document}